\documentclass[11pt,oneside]{amsart}
\usepackage{amsmath,amsthm,amsfonts,amssymb, ulem}
\usepackage{hyperref}

\usepackage {tikz}
\usepackage{comment}

\usepackage{wrapfig}

\usetikzlibrary {positioning}
\definecolor {processblue}{cmyk}{0.96,0,0,0}

\newtheorem{theorem}{Theorem}[section]
\newtheorem{lemma}[theorem]{Lemma}
\newtheorem{proposition}[theorem]{Proposition}

\newtheorem{remark}[theorem]{Remark}

\theoremstyle{definition}

\usepackage{amstext}
\usepackage{array}
\newcolumntype{L}{>{$}l<{$}}

\def\MN{{\mathbb{N}}}

\def\CC{{\mathcal{C}}}

\def\CP{{\mathcal{P}}}
\def\CR{{\mathcal{R}}}
\def\CT{{\mathcal{T}}}

\DeclareMathOperator{\CESP}{\textbf{CPP}}
\DeclareMathOperator{\Aut}{{\bf Aut}}
\DeclareMathOperator{\St}{{St}}
\DeclareMathOperator{\red}{{red}}

\DeclareMathOperator{\lift}{{lift}}
\DeclareMathOperator{\CPr}{\textbf{CP}}
\DeclareMathOperator{\Ch}{{Ch}}

\newcommand{\gp}[1]{{\left\langle #1 \right\rangle}}

\newcommand{\Set}[2]{\left\{\, #1 \;\middle|\; #2 \,\right\}}
\newcommand{\rb}[1]{{\left( #1 \right)}}

\usepackage{graphicx}

\title[Conjugacy problem in the first Grigorchuk group]{Linear time algorithm for the conjugacy problem in the first Grigorchuk group}

\author{Mitra Modi, Mathew Seedhom, Alexander Ushakov}
\address{Department of Mathematics, Stevens Institute of Technology, Hoboken, NJ, USA}
\curraddr{}
\email{mmodi3,mseedhom,aushakov@stevens.edu}
\thanks{The authors are thankful for the support from the Pinnacle Scholars Program of Stevens Institute of Technology}

\date{\today}

\begin{document}

\begin{abstract}
We prove that the conjugacy problem in
the first Grigorchuk group $\Gamma$
can be solved in linear time. Furthermore, the problem to decide
if a list of elements $w_1,\ldots,w_k\in\Gamma$ contains a pair
of conjugate elements can be solved in linear time.
We also show that
a conjugator for a pair of conjugate element $u,v\in\Gamma$
can be found in polynomial time.

\end{abstract}

\maketitle

\section{Introduction}

The first Grigorchuk group $\Gamma$ is famous for its
numerous remarkable properties. It is an infinite
residually finite group.
$\Gamma$ is \emph{just infinite}, i.e., $\Gamma$
is infinite but every proper quotient group of $\Gamma$ is finite.
Every element in $\Gamma$ has finite order that is a power of $2$,
see \cite{Grigorchuk:1980}.
$\Gamma$ has \emph{intermediate growth}, i.e., it
has growth that is faster than any
polynomial but slower than any exponential function,
see \cite{Grigorchuk:1984}.
$\Gamma$ is \emph{amenable},
but not \emph{elementary amenable}, see \cite{Grigorchuk:1984}.
It is finitely generated, but is not finitely related,
see \cite{Lysenok:1985}.
It has decidable word problem, conjugacy problem \cite{Leonov:1998,Rozhkov:1998},
and subgroup membership problem \cite{Grigorchuk-Wilson:2003}.
The Diophantine problem for quadratic equations is
decidable, see \cite{Lysenok_Miasnikov_Ushakov:2016}.
$\Gamma$ has commutator width $2$, see \cite{Bartholdi-Groth-Lysenok:2019}.
Also, $\Gamma$ was proposed as a possible platform for cryptographic
schemes, see \cite{GZ,P,MU5}.
These and other properties of $\Gamma$ are discussed
in detail in \cite{Harpe} and \cite{Grigorchuk:2006}.
In this paper we discover a new special property of $\Gamma$.
Consider the following problems.

\medskip\noindent
\textbf{Conjugacy problem in $G$,
abbreviated $\CPr(G)$},
is an algorithmic question to determine if
for given $w_1,w_2\in G$ there exists
$x\in G$ such that $w_1 = x^{-1} w_2 x$,
or not.

\medskip\noindent
\textbf{Conjugate pair problem in $G$,
abbreviated $\CESP(G)$},
is an algorithmic question to determine if
a list of elements $w_1,\ldots,w_k\in G$
contains a pair of conjugate elements,
or not.

\medskip
We prove that $\CESP(\Gamma)$ can be solved in linear time.
In particular, the conjugacy problem
and, hence, the word problem can be solved in linear time
in $\Gamma$.
This is surprising because
there are not many problems in computational algebra
that can be solved in linear time. In particular,
the class of groups that have linear conjugacy problem
appears to be small.
Below we review several types of groups
that have (nearly) linear conjugacy problem.

It is easy to see that \emph{virtually abelian} groups
have linear time $\CESP$.
\emph{Free non-abelian} groups have linear time $\CESP$ too.
Indeed, for $w_1,\ldots,w_k$
\begin{itemize}
\item
cyclically reduce $w_i$'s;
\item
using \emph{Booth's algorithm},
see \cite{Booth:1980},
compute lexicographically minimal cyclic permutations of $w_i$'s;
\item
decide if two lexicographically minimal cyclic permutations are equal
(e.g. using tries described in Section \ref{se:trie}).
\end{itemize}
D. Epstein and D. Holt in \cite{Epstein-Holt:2006}
proved that
the conjugacy problem can be solved
in linear time in any fixed \emph{word-hyperbolic} group.
Their algorithm reduces an instance of the conjugacy problem
to an instance of the pattern matching problem
(Knuth--Morris--Pratt algorithm, see \cite{KMP}).
We believe that
using Booth's algorithm instead of KMP
Epstein--Holt algorithm can be improved
to solve $\CESP$ in linear time.
For a fixed finitely generated
\emph{nilpotent} group the best known result states that
the conjugacy problem can be solved in nearly linear time
$O((|u|+|v|)\log^2(|u|+|v|))$ for an instance
$(u,v)$, see
\cite[Theorem 4.6]{Macdonald-Myasnikov-Nikolaev-Vassileva:2015}.
It is not known to the authors if that can be improved
to linear time complexity especially for a list of elements.
The conjugacy problem in other classes of groups
appears to be harder.

Notice that for each class of groups with linear time $\CESP$
mentioned above a decision algorithm
for an instance $w_1,\ldots,w_k$ of the problem
computes a list $w_1^\ast,\ldots,w_k^\ast$ of
unique representatives of the corresponding conjugacy classes,
and decides if the list contains two equal objects.
Our new algorithm for the Grigorchuk group $\Gamma$
acts in a similar fashion.
Given a list of words $w_1,\ldots,w_k$ it
assigns a representative $w_i^\ast$ for the
conjugacy class of $w_i$.
Our choice of $w_i^\ast$ is not unique
and depends on a given set $\{w_1,\ldots,w_k\}$.

In Section \ref{se:preliminaries} we review the definition
of the group $\Gamma$, the splitting operation, and the norm.
In Section \ref{se:conjugacy-gamma} we present the
linear time algorithm for $\CESP(\Gamma)$.
In Section \ref{se:conjugator} we prove that
a conjugator for a pair of conjugate element $u,v\in\Gamma$
can be found in polynomial time.

\section{Preliminaries: the Grigorchuk group $\Gamma$}
\label{se:preliminaries}

\subsection{Automorphisms on an infinite regular rooted binary tree}

An infinite regular rooted binary tree $\CT$ is
defined by a pair $(V, E)$,
with the vertex set $V = \{0, 1\}^\ast$ and the edge set
$E = \{w \to w0 \mid w \in \{0, 1\}^\ast\} \cup \{w \to w1 \mid w \in \{0, 1\}^*\}$,
where $w0$ and $w1$ refer to the concatenations of $w$ with $0$ and $w$ with $1$ respectively.
An automorphism of $\CT$ is a bijective function $\alpha:V\to V$ that preserves connectedness, i.e., satisfies
$\alpha(w)$ is connected to $\alpha(w')$
if and only if $w$ is connected to $w'$.
The set of all automorphisms, denoted by $\Aut(\CT)$,
with composition is a group called the
\emph{group of automorphisms} of $\CT$.

\subsection{The group $\Gamma$}

The first Grigorchuk group $\Gamma$ is a subgroup of
$\Aut(\CT)$ generated by elements $a,b,c,d$ defined
as follows.
For every $w\in\{0,1\}^\ast$
\begin{align*}
a(0w)=&1w, & a(1w)=&0w,\\
b(0w)=&0 a(w),& b(1w)=&1 c(w),\\
c(0w)=&0 a(w),& c(1w)=&1 d(w),\\
d(0w)=&0w,& d(1w)=&1 b(w),
\end{align*}
which can be visualized as shown in the figure below.
\begin{center}
\includegraphics{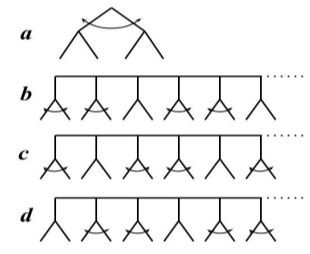}
\end{center}
The elements $a,b,c,d$ satisfy the following identities in $\Gamma$:
\begin{align*}
& a^2=b^2=c^2=d^2=1 && cd=dc=b\\
& bc=cb=d && bd=db=c,
\end{align*}
and, hence, every word $w$ over the group alphabet
$X = \{a,b,c,d\}$ of $\Gamma$
can be reduced to a word $\red(w)$ of the form
$$
[a]\ast_1 a\ast_2 a \ldots a \ast_k [a],
$$
where $[a]$ stands for $a$ or $\varepsilon$ and
$\ast_i \in \{b,c,d\}$, called a \emph{reduced word}.
Denote the set of all reduced words by $\CR$.

Consider the subgroup
$\St(1) = \{g\in\Gamma \mid g(0)=0 \mbox{ and } g(1)=1\}$,
called \emph{the first level stabilizer}.
It is easy to see that $w\in\St(1)$
if and only if
$\delta_a(w)$ (the number of occurrences of $a$ in $w$) is even,
and
$$
\St(1)=\gp{b, c, d, aba, aca, ada }.
$$
Denote the set $\Set{w\in\CR}{\delta_a(w)\mbox{ is even}}$ by $\CR^e$.

\subsection{Splitting}

If $g\in \St(1)$, then, by definition, $g$ fixes vertices $0$ and $1$
and naturally defines automorphisms $g_0,g_1$
on the left and right subtrees of $\CT$ respectively.
This defines a homomorphism $\psi: \St(1) \to \Gamma \times \Gamma$
\begin{center}
$g\stackrel{\psi}{\mapsto}(g_0,g_1)$,
\end{center}
and homomorphisms into each component
$$
g\stackrel{\phi_0}{\longmapsto} g_0
\mbox{ and }
g\stackrel{\phi_1}{\longmapsto} g_1.
$$
We often write
$g=(g_0,g_1)$ instead of $g=\psi(g_0,g_1)$
to simplify the notation.
Since $\St(1)$ is generated by $b,c,d,aba,aca,ada$
to define $\psi$ it sufficient to define the images of $\psi$
on the generators:
\begin{align*}
\psi(b) &= (a, c)  & \psi(aba) &= (c, a)\\
\psi(c) &= (a, d)  & \psi(aca) &= (d, a)\\
\psi(d) &= (1, b)  & \psi(ada) &= (b, 1).
\end{align*}

\subsection{Norm}

Here we shortly recall the results of
\cite[Section 3]{Lysenok_Miasnikov_Ushakov:2010}
(see also \cite{Bartholdi:1998}).
Fix positive real values $\gamma_a$, $\gamma_b$, $\gamma_c$, $\gamma_d$,
termed \emph{weights}.
For a word $w \in X^*$ the number
$$
||w|| = \gamma_a\delta_a(w)+\gamma_b\delta_b(w)+\gamma_c\delta_c(w)+\gamma_d\delta_d(w),
$$
is called the \emph{norm} of $w$. The length $|w|$ is
a special case of the norm when
$\gamma_a$ = $\gamma_b$ = $\gamma_c$ = $\gamma_d$ = 1.
Clearly, for any $u, v, w \in X^\ast$ the following holds:
\begin{itemize}
\item $\min\{\gamma_a, \gamma_b, \gamma_c, \gamma_d\}\dot|w| \le ||w|| \le \max\{\gamma_a, \gamma_b, \gamma_c, \gamma_d\}\dot|w|$.
\item $||uv|| = ||u|| + ||v||$.
\item If the numbers $\gamma_b, \gamma_c, \gamma_d$ satisfy the triangle inequality, then
$$
||red(w)|| \le ||w||.
$$
\end{itemize}

Next, we define particular
$\gamma_a, \gamma_d, \gamma_c, \gamma_b$ to be
used for the rest of the paper.
Let $\alpha$ be the unique
real root of the polynomial $2x^3 - x^2 - x - 1$,
$$
\alpha \in (1.233751, 1.233752) \approx 1.23375.
$$
Put
\begin{align*}
\gamma_a &= \alpha^2 + \alpha  - 1 &&\in (1.755892, 1.755896),\\
\gamma_b &= 2,\\
\gamma_c &= \alpha^2 - \alpha + 1 &&\in (1.28839, 1.288392),\\
\gamma_d &= -\alpha^2 + \alpha + 1 &&\in (0.711608, 0.71161).
\end{align*}
It is easy to check that
the weights $\gamma_a, \gamma_d, \gamma_c, \gamma_b$
satisfy the triangle inequality and for any $w\in\CR$ we have
\begin{equation}\label{eq:norm_properties}
|w|<\frac{1}{0.7}||w||
\ \ \ \mbox{ and }\ \ \
||w|| \le 2 |w|.
\end{equation}
For $w \in \CR$ and $i = 0, 1$ define
$$
w_i =
\begin{cases}
red(\phi_i(w)),  & \mbox{if } w \in \mathcal{R}^e, \\
red(\phi_i(wa)), & \mbox{if } w \notin \mathcal{R}^e.
\end{cases}
$$

\begin{proposition}[{{\cite[Proposition 3.6]{Lysenok_Miasnikov_Ushakov:2010}}}]
\label{pr:norm-bound}
For any $w \in \CR$ the following holds:
\begin{itemize}
\item[(a)]
If $||w|| \ge 9$, then $\frac{||w||}{||w_0|| + ||w_1||} \ge 1.03.$
\item[(b)]
If $||w|| \ge 200$, then $\frac{||w||}{||w_0|| + ||w_1||} \ge 1.22.$
\end{itemize}
\end{proposition}

\section{Conjugacy problem in $\Gamma$}
\label{se:conjugacy-gamma}

Let $G$ be a group and $u, v \in G$. We say that $u$ and $v$ are \emph{conjugate} if
$u=x^{-1}vx$ for some $x \in G$, denoted $u\sim_G v$.
\emph{The conjugacy problem in $G$}, denoted $\CPr(G)$, is an algorithmic question
of deciding if two given elements $u,v\in G$ are  conjugate, or not.
In this section we describe a linear time algorithm
for the conjugacy problem in $\Gamma$.
But first we need to recall several facts
from \cite{Grigorchuk:2006}.

\begin{proposition}[{{\cite{Grigorchuk:2006}}}]
\label{pr:conj-systems}
For any $u, v, x \in X^\ast$ the following holds:
\begin{enumerate}
    \item If $u, v, x \in \St(1)$, and $u=(u_0, u_1)$, $v=(v_0, v_1)$, $x=(x_0, x_1)$, then
    \begin{center}
        $u=x^{-1}vx \iff $
        $\begin{cases}$
        $u_0=x_0^{-1}v_0x_0$, $\\$
        $u_1=x_1^{-1}v_1x_1$.
        $\end{cases}$
    \end{center}
    \item If $u, v, xa \in \St(1)$, and $u=(u_0, u_1)$, $v=(v_0, v_1)$, $xa=(x_0, x_1)$, then
    \begin{center}
        $u=x^{-1}vx \iff $
        $\begin{cases}$
        $u_0=x_1^{-1}v_1x_1$, $\\$
        $u_1=x_0^{-1}v_0x_0$.
        $\end{cases}$
    \end{center}
    \item If $ua, va, x \in \St(1)$, and $ua=(u_0, u_1)$, $va=(v_0, v_1)$, $x=(x_0, x_1)$, then
    \begin{center}
        $u=x^{-1}vx \iff $
        $\begin{cases}$
        $u_0u_1=x_0^{-1}v_0v_1x_0$, $\\$
        $x_1=v_1x_0u_1^{-1}$.
        $\end{cases}$
    \end{center}
    \item If $ua, va, xa \in \St(1)$, and $ua=(u_0, u_1)$, $va=(v_0, v_1)$, $xa=(x_0, x_1)$, then
    \begin{center}
        $u=x^{-1}vx \iff $
        $\begin{cases}$
        $u_1u_0=x_0^{-1}v_0v_1x_0$, $\\$
        $x_1=v_1x_0u_0^{-1}$.
        $\end{cases}$
    \end{center}
\end{enumerate}
\end{proposition}

\subsection{$Q$-sets}

Let $K$ be the normal closure of the element $abab$ in $\Gamma$.
Some important properties of $K$ related to the conjugacy problem
are described in
\cite[Section 2.3]{Lysenok_Miasnikov_Ushakov:2010}. In particular, $[\Gamma:K]=16$ and
there is a set $L\subseteq\Gamma/K\times\Gamma/K$
such that a pair $(w_0,w_1)\in\Gamma\times \Gamma$
has a $\psi$-preimage if and only if $(w_0K,w_1K)\in L$.
Furthermore,
for any $w=(w_0,w_1),w'=(w_0',w_1')$,
$$
(w_0K,w_1K)=(w_0'K,w_1'K) \in L
\ \ \Rightarrow \ \
wK=w'K.
$$
Hence, we can naturally define
a ``lifting'' function for $K$-cosets
$$
\lift:\CP(\Gamma/K\times\Gamma/K)\to\CP(\Gamma/K),
$$
where $\CP$ stands for a power-set as follows:
$$
\lift(S) = \Set{ wK }{w=(w_0,w_1) \mbox{ and } (w_0K,w_1K) \in S},
$$
for $S\subseteq \Gamma/K\times\Gamma/K$.
See \cite[Section 2.3]{Lysenok_Miasnikov_Ushakov:2010} for more detail.

For $u, v \in \Gamma$ define a (finite) set
$$
Q(u, v) = \Set{xK}{ u = x^{-1}vx } \subseteq \Gamma/K.
$$
Clearly, $u$ and $v$ are conjugate in $\Gamma$ if and only if $Q(u,v) \ne \emptyset$.
The next proposition
follows from Proposition \ref{pr:conj-systems}
and the definitions of $K$ and $\lift$.

\begin{proposition}[Cf. {{\cite[Lemma 4.3]{Lysenok_Miasnikov_Ushakov:2010}}}]
\label{pr:lifting}
For any $u,v \in X^*$ the following holds.
If $u=(u_0,u_1), v=(v_0,v_1) \in \St(1)$, then
\begin{equation}\label{eq:even-Q-formula}
Q(u,v)=
\lift[Q(u_0,v_0) \times Q(u_1,v_1)] \cup
\lift[Q(u_1,v_0) \times Q(u_0,v_1)]a.
\end{equation}
If $ua=(u_0,u_1), va=(v_0,v_1) \in \St(1)$, then
\begin{align}
\label{eq:odd-Q-formula}
Q(u,v)=&\lift\Set{(g, v_1 g u^{-1}_1)}{g \in Q(u_0u_1, v_0v_1)} \\
\cup & \lift\Set{(gu_1^{-1}, v_0^{-1}g)}{g \in Q(u_0u_1, v_0v_1)}a \nonumber.
\end{align}
Furthermore, it takes constant time to compute
$Q(u,v)$ in \eqref{eq:even-Q-formula} and
\eqref{eq:odd-Q-formula}
provided $u_0K,u_1K,v_0K,v_1K$ and
the $Q$-sets involved on the right-hand side.
\end{proposition}

\subsection{Splitting tree $T(w)$}
\label{se:splitting_tree}

Proposition \ref{pr:lifting} reduces the problem
of computing $Q(u,v)$ to the problem
of computing the following sets:
\begin{itemize}
\item
$Q(u_0,v_0),Q(u_0,v_1),Q(u_1,v_0),Q(u_1,v_1)$ if $u,v\in\St(1)$;
\item
$Q(u_0u_1,v_0v_1)$ if $u,v\notin\St(1)$.
\end{itemize}
This naturally defines trees $T(u)$ and $T(v)$,
called \emph{splitting trees}.
The \emph{splitting tree} $T(w)$
for a given $w\in\CR$
is a directed tree $(V,E)$
equipped with a labeling function $\mu:V\to\CR$,
with the root denoted by $root(T(w))\in V$, satisfying
the following conditions:
\begin{itemize}
\item
$\mu(root(T(w)))=w$;
\item
if $|w|\le 1$, then $T(w)$ has a single vertex;
\item
if $w=(w_0,w_1)\in\St(1)$, then $T(w)$
is obtained by connecting a single vertex
labeled with $w$
to the roots of $T(w_0)$ and $T(w_1)$;
\item
if $w=(w_0,w_1)a\notin\St(1)$, then $T(w)$
is obtained by connecting a single vertex
labeled with $w$
to the root of $T(w_0w_1)$.
\end{itemize}
Define $T_9(w)$ to be a subtree of $T(w)$
induced by a subset of vertices
$$
\{v\in V(T(w))\mid \|\mu(v)\|\ge 9\}.
$$
It is straightforward to check that $T_9(w)$ is connected
using Table \ref{tab:w9}.
Define the \emph{total norm} of a labelled graph $T$
to be
$$
\|T\| = \sum_{v\in V(T)} \|\mu(v)\|.
$$
Denote by $\Ch_T(v)=\Ch_T^{(1)}(v) \subseteq V(T)$
the set of children of a vertex $v$ in a
labelled graph $T$.
Inductively, define $\Ch_T^{(k)}(v)$ to be the set of children
of vertices from $\Ch_T^{(k-1)}(v)$.

\begin{table}[h]
\centering
\begin{tabular}{|l|l|l||l|l|l||l|l|l|}
\hline
$w$ & $\|w\|$ & desc&
$w$ & $\|w\|$ & desc&
$w$ & $\|w\|$ & desc\\
\hline
\hline
$\varepsilon$ & 0 &   & $adadad$ & 7.4037 & $b$ & $cadac$ & 6.7998 & $aba$, $\varepsilon$ \\
\hline
$a$ & 1.7559 &   & $b$ & 2 &   & $cadaca$ & 8.5557 & $aba$ \\
\hline
$ab$ & 3.7559 & $ca$ & $ba$ & 3.7559 & $ac$ & $cadad$ & 6.2238 & $ab$, $c$ \\
\hline
$aba$ & 5.5118 & $c$, $a$ & $bab$ & 5.7559 & $\varepsilon$ & $cadada$ & 7.9797 & $ad$ \\
\hline
$abab$ & 7.5118 & $ca$, $ac$ & $baba$ & 7.5118 & $ac$, $ca$ & $cadadad$ & 8.6917 & $ac$ \\
\hline
$abac$ & 6.7998 & $ca$, $ad$ & $babac$ & 8.7998 & $aca$, $cad$ & $d$ & 0.712 &   \\
\hline
$abaca$ & 8.5557 & $b$ & $babad$ & 8.2238 & $ac$, $cab$ & $da$ & 2.4679 & $b$ \\\hline
$abad$ & 6.2238 & $c$, $ab$ & $bac$ & 5.0439 & $aba$ & $dab$ & 4.4679 & $da$ \\
\hline
$abada$ & 7.9797 & $cab$ & $baca$ & 6.7998 & $ad$, $ca$ & $daba$ & 6.2238 & $c$, $ba$ \\
\hline
$abadad$ & 8.6917 & $dab$ & $bacab$ & 8.7998 & $ada$, $cac$ & $dabab$ & 8.2238 & $ca$, $bac$ \\
\hline
$ac$ & 3.0439 & $da$ & $bacac$ & 8.0878 & $ada$, $cad$ & $dabac$ & 7.5118 & $ca$, $bad$ \\
\hline
$aca$ & 4.7998 & $d$, $a$ & $bacad$ & 7.5118 & $ad$, $cab$ & $dabad$ & 6.9358 & $c$, $bab$ \\
\hline
$acab$ & 6.7998 & $da$, $ac$ & $bad$ & 4.4679 & $ad$ & $dabada$ & 8.6917 & $dab$ \\
\hline
$acaba$ & 8.5557 & $b$ & $bada$ & 6.2238 & $ab$, $c$ & $dac$ & 3.7559 & $ca$ \\
\hline
$acac$ & 6.0878 & $da$, $ad$ & $badab$ & 8.2238 & $aba$, $\varepsilon$ & $daca$ & 5.5118 & $d$, $ba$ \\
\hline
$acaca$ & 7.8437 & $\varepsilon$ & $badac$ & 7.5118 & $aba$, $b$ & $dacab$ & 7.5118 & $da$, $bac$ \\
\hline
$acacad$ & 8.5557 & $dabad$ & $badad$ & 6.9358 & $ab$, $d$ & $dacac$ & 6.7998 & $da$, $bad$ \\
\hline
$acad$ & 5.5118 & $d$, $ab$ & $badada$ & 8.6917 & $ac$ & $dacaca$ & 8.5557 & $dabad$ \\
\hline
$acada$ & 7.2677 & $dab$ & $c$ & 1.288 &   & $dacad$ & 6.2238 & $d$, $bab$ \\
\hline
$acadac$ & 8.5557 & $aba$ & $ca$ & 3.0439 & $ad$ & $dacada$ & 7.9797 & $cab$ \\
\hline
$acadad$ & 7.9797 & $cab$ & $cab$ & 5.0439 & $aba$ & $dacadad$ & 8.6917 & $dab$ \\
\hline
$ad$ & 2.4679 & $b$ & $caba$ & 6.7998 & $ac$, $da$ & $dad$ & 3.1799 & $\varepsilon$ \\
\hline
$ada$ & 4.2238 & $b$, $\varepsilon$ & $cabab$ & 8.7998 & $aca$, $dac$ & $dada$ & 4.9358 & $b$, $b$ \\
\hline
$adab$ & 6.2238 & $ba$, $c$ & $cabac$ & 8.0878 & $aca$, $dad$ & $dadab$ & 6.9358 & $ba$, $d$ \\
\hline
$adaba$ & 7.9797 & $bac$ & $cabad$ & 7.5118 & $ac$, $dab$ & $dadaba$ & 8.6917 & $bad$ \\
\hline
$adabad$ & 8.6917 & $bad$ & $cac$ & 4.3319 & $\varepsilon$ & $dadac$ & 6.2238 & $ba$, $c$ \\
\hline
$adac$ & 5.5118 & $ba$, $d$ & $caca$ & 6.0878 & $ad$, $da$ & $dadaca$ & 7.9797 & $bac$ \\
\hline
$adaca$ & 7.2677 & $bad$ & $cacab$ & 8.0878 & $ada$, $dac$ & $dadacad$ & 8.6917 & $bad$ \\
\hline
$adacac$ & 8.5557 & $b$ & $cacac$ & 7.3758 & $ada$, $dad$ & $dadad$ & 5.6478 & $b$, $\varepsilon$ \\
\hline
$adacad$ & 7.9797 & $bac$ & $cacad$ & 6.7998 & $ad$, $dab$ & $dadada$ & 7.4037 & $b$ \\
\hline
$adad$ & 4.9358 & $b$, $b$ & $cacada$ & 8.5557 & $b$ & $dadadac$ & 8.6917 & $ca$ \\
\hline
$adada$ & 6.6917 & $\varepsilon$ & $cad$ & 3.7559 & $ac$ & $dadadad$ & 8.1157 & $\varepsilon$ \\
\hline
$adadab$ & 8.6917 & $ca$ & $cada$ & 5.5118 & $ab$, $d$ & & & \\
\hline
$adadac$ & 7.9797 & $da$ & $cadab$ & 7.5118 & $aba$, $b$ & & & \\
\hline
\end{tabular}
\caption{Words $w$ satisfying $\|w\|<9$ and their immediate descendants.}
\label{tab:w9}
\end{table}

\begin{lemma}\label{le:total-norm_conj}
For any $w\in\CR$ the following holds:
\begin{itemize}
\item[(a)]
If $\|w\|< 9$, then $\|T(w)\| < 30$.
\item[(b)]
If $\|w\|\ge 9$, then $\|T_9(w)\| \le 35 \|w\|$.
\item[(c)]
If $\|w\|\ge 9$, then $|V(T_9(w))| \le 4 \|w\|$.
\item[(d)]
$\|T(w)\| \le 275 \|w\|$.
\item[(e)]
$\sum_{v\in V(T)} |\mu(v)| \le 800|w|.$
\end{itemize}
\end{lemma}

\begin{proof}
Table \ref{tab:w9} enumerates all words $w$ satisfying $\|w\|<9$
and their immediate descendants in $T(w)$.
It is straightforward to check that
$\|T(w)\| < 30$ holds for every word in the table.

Next, it follows from Proposition \ref{pr:norm-bound} that
for every vertex $v\in T_9(w)$
$$
\sum_{v' \in \Ch_{T_9(w)}(v)} \|\mu(v')\|
\le \frac{||\mu(v)||}{1.03},
$$
and for every $k\in\MN$
$$
\sum_{v' \in \Ch_{T_9(w)}^{(k)}(v)} \|\mu(v')\|
\le \frac{||\mu(v)||}{1.03^k}.
$$
Hence, the total norm of $T_9(w)$ is bounded above by
$$
\sum_{i=0}^{\infty} \frac{||w||}{1.03^i} = ||w||\frac{1}{1-\frac{1}{1.03}}
= \frac{1.03}{0.03}\cdot ||w|| \le 35\|w\|,
$$
and, hence, (b) holds.
Each vertex in $T_9(w)$ is labelled with a word of norm greater than or equal to $9$.
Hence,
$$
|V(T_9(w))| \le \tfrac{35}{9} \|w\| < 4\|w\|,
$$
and (c) holds.
Every vertex $v\in T_9(w)$ has at most two descendants
of norm less than $9$.
Therefore,
\begin{align*}
\|T(w)\|
&\le
\sum_{v\in V(T_9(w))} \rb{\|\mu(v)\| + 2\cdot 30}
\le
\|T_9(w)\| +  60 |V(T_9(w))|\\
&\le 35\|w\| + 240\|w\| = 275\|w\|,
\end{align*}
and (d) holds.
Finally, it follows from
inequalities \eqref{eq:norm_properties} and item (d)
that
$$
\sum_{v\in V(T)} |\mu(v)|
\le \sum_{v\in V(T)} \tfrac{1}{0.7}\|\mu(v)\|
\le \frac{275}{0.7}\|w\|
< 400\|w\|
\le 800|w|,
$$
and, hence, (e) holds.
\end{proof}

\begin{proposition}\label{pr:descendant_length}
For any edge $u\to v$ in $T(w)$ we have
\begin{equation}\label{eq:ineq1}
|\mu(u)| \ge |\mu(v)|.
\end{equation}
Furthermore, for any path
$u\to u'\to u''\to u'''$ in $T(w)$ we have
\begin{equation}\label{eq:ineq2}
|\mu(u)| > |\mu(u''')|.
\end{equation}
\end{proposition}

\begin{proof}
Denote $\mu(u)$ by $w$ and consider two cases.
If $w \in \{\varepsilon,a,b,c,d\}$, then
$u$ has no children.
If $w\in\CR^e \setminus\{\varepsilon,a,b,c,d\}$, then
$|w|\ge 3$, $u$ has two children $u_0,u_1$ satisfying
$$
|\mu(u_0)|,|\mu(u_1)|
\le
\left\lfloor \frac{|w|+1}{2} \right\rfloor
<|w|
$$
and the statement follows.

If $w\in\CR\setminus\CR^e\setminus\{\varepsilon,a,b,c,d\}$,
then $w = (w_0,w_1)a$,
$u$ has a single child labeled with $w_0w_1$, and
$w$ belongs to one of the sets defined by the following
regular expressions:
\begin{itemize}
\item[(r1)]
$w\in a ((b|c|d)a(b|c|d)a)^\ast$, in which case
$|w| = 4n+1$ and $|w_0w_1|\le 4n$;
\item[(r2)]
$w\in a (b|c|d) (a(b|c|d)a(b|c|d))^\ast$, in which case
$|w| = 4n+2$ and $|w_0w_1|\le 4n+2$;
\item[(r3)]
$w\in (b|c|d) a ((b|c|d)a(b|c|d)a)^\ast$, in which case
$|w| = 4n+2$ and $|w_0w_1|\le 4n+2$;
\item[(r4)]
$w\in (b|c|d) a (b|c|d) (a(b|c|d)a(b|c|d))^\ast$, in which case
$|w| = 4n+3$ and $|w_0w_1|\le 4n+3$
because $w_0$ ends with a letter in $\{b,c,d\}$
and $w_1$ starts with a letter in $\{b,c,d\}$.
\end{itemize}
Hence, in all cases, \eqref{eq:ineq1} holds.

Similar analysis proves the following.
If $\mu(u)$ contains the letter $d$,
then $|\mu(u)| > |\mu(v)|$.
If $w$ does not contain $d$, but contains $c$,
and $|\mu(u)| = |\mu(v)|$, then $v$ contains $d$
and, hence, children of $v$ have shorter labels.
Same way we can analyse words $w$ that contain
$a$'s and $b$'s only.
\end{proof}

\subsection{Lifting conjugacy classes}

For $v\in\Gamma$ by $c_v$ we denote the conjugacy class of $v$
$$
c_v=\Set{v'\in\Gamma}{v'\sim_\Gamma v},
$$
and by $c$ the corresponding map $v\stackrel{c}{\mapsto} c_v$.
Also, for $v\in\Gamma$ we define a set
$P_v \subseteq \Gamma$ as follows:
$$
P_v =
\begin{cases}
\Set{u=(u_0,u_1)}{u_0\sim_\Gamma v_0,\ u_1\sim_\Gamma v_1}& \mbox{if }v=(v_0,v_1)\in\St(1),\\
\Set{u=(u_0,u_1)a}{u_0u_1\sim_\Gamma v_0v_1}& \mbox{if }v=(v_0,v_1)a\notin\St(1).
\end{cases}
$$
Theorem \ref{th:diff-conj-classes} implies
that for any $v\in\Gamma$ the conjugacy classes
of its children
almost uniquely define the conjugacy class of $v$.

\begin{theorem}\label{th:diff-conj-classes}
$|c(P_v)| \le 256$ for every $v\in\Gamma$.
\end{theorem}

\begin{proof}
For $u,v\in\Gamma$ define a set
$T(u,v)\subseteq \Gamma/K\times \Gamma/K$
as follows:
\begin{equation}\label{eq:T}
\begin{cases}
Q(u_0,v_0)\times Q(u_1,v_1), & \mbox{if } u=(u_0,u_1), v=(v_0,v_1), \\
\Set{(gK,v_1gu_1^{-1}K) }{ g\in Q(u_0u_1,v_0v_1)} & \mbox{if } u=(u_0,u_1)a, v=(v_0,v_1)a,\\
\emptyset&\mbox{otherwise.}
\end{cases}
\end{equation}
If $u,v\in\St(1)$, then $T(u,v)$ is a set
of solutions (modulo $K$) of the system
\begin{equation}\label{eq:even-by-even-system}
\left\{
\begin{array}{l}
u_0 = x_0^{-1} v_0 x_0, \\
u_1 = x_1^{-1} v_1 x_1.
\end{array}
\right.
\end{equation}
If $u,v\notin\St(1)$, then $T(u,v)$
is a set of solutions (modulo $K$) of the system
\begin{equation}\label{eq:odd-by-even-system}
\left\{
\begin{array}{l}
u_0 = x_0^{-1} v_0 x_1, \\
u_1 = x_1^{-1} v_1 x_0.
\end{array}
\right.
\end{equation}
Hence,
for any $u,v,w\in\St(1)$
or for any $u,v,w\notin\St(1)$
the following holds:
$$T
(u,v) = T(v,u)^{-1}
\ \ \mbox{ and } \ \
T(v,w)T(u,v) \subseteq T(u,w).
$$
Therefore, for any $v\in\Gamma$ and
$\{u_1.\ldots,u_{257}\} \subseteq P_v$
there are distinct $i,j$ such that
$T(v,u_i)\cap T(v,u_j)\ne\emptyset$, which
implies that $(1,1) \in T(u_i,u_j)$.
Thus, by Proposition \ref{pr:lifting},
$u\sim_\Gamma v$.
\end{proof}

\subsection{Trie}\label{se:trie}

Let us recall definition of a tree data structure called a \emph{trie}.
Consider a rooted tree $\Lambda = (V,E)$ with an edge labeling
function $\nu:E\to\{a,b,c,d\}$
and a function $m:V\to\{0,1\}$
that defines a subset of \emph{marked vertices} $\{v\in V\mid m(v)=1\}$.
The triple $(\Lambda,\nu,m)$ defines a set of words
$S(\Lambda,\nu,m)$
that can be read
starting from the root to the marked vertices.
We assume that for every vertex $v\in V$ and any $l\in\{a,b,c,d\}$
there is at most one edges leaving $v$ labeled with $l$, which
excludes multisets. The following lemma is obvious.

\begin{lemma}\label{le:array_of-words-add}
For any trie $(\Lambda,\nu,m)$ and a word $w$
it takes $O(|w|)$ time
\begin{itemize}
\item
to decide if $w\in S(\Lambda,\nu,m)$,
\item
to construct $(\Lambda',\nu',m')$ such that
$S(\Lambda',\nu',m') = S(\Lambda,\nu,m) \cup \{w\}$.
\end{itemize}
\qed
\end{lemma}

\begin{lemma}\label{le:array_of-words}
For any words $u_1,\ldots,u_k\in\CR$
it takes $O(|u_1|+\ldots+|u_k|)$ time to construct
$(\Lambda,\nu,m)$ such that
$S(\Lambda,\nu,m) = \{u_1,\ldots,u_k\}$.
\qed
\end{lemma}

\subsection{Representatives of conjugacy classes}
\label{se:representatives}

Let $T(w) = (V,E)$ be the splitting tree for $w$.
Define sets
\begin{align*}
W &= \{\mu(v) \mid v\in V\} \cup \{\varepsilon,a,b,c,d\}, \\
P&=\{P_w \mid w\in W\},\\
\CC&=\{c_{w}\mid w\in W\}.
\end{align*}
By definition, $P_w$ is uniquely defined
by a pair $(c_{w_0},c_{w_1})$ if $w\in\St(1)$
and by $c_{w_0w_1}$ if $w\notin\St(1)$.
To efficiently work with conjugacy classes,
for each class $c_w\in\CC$ we choose a particular representative $w^\ast\in W$
and organize representatives in a $|P|\times 256$ table.
Each row corresponds to some $P_w$, contains
representatives of conjugacy classes in $c(P_w)$, and
is labelled with a pair $(w_0^\ast,w_1^\ast)$ if $w\in\St(1)$
or with $(w_0 w_1)^\ast$ if $w\notin\St(1)$.
Furthermore, for each $w\in W$ we keep
\begin{itemize}
\item
a reference to its representative $w^\ast$,
\item
a set $Q(w,w^\ast)$.
\end{itemize}
Initially, the table contains five rows corresponding to elements
$1,a,b,c,d$.
\begin{itemize}
\item
A row for $P_\varepsilon$ is defined by a pair $(\varepsilon,\varepsilon)$ and contains a single word $\varepsilon$ which represents the conjugacy class $c_\varepsilon$.
\item
A row for $P_a$ is defined by $\varepsilon$ and contains
a single word $a$ which represents the conjugacy class $c_a$.
\item
A row for $P_b$ is defined by a pair $(a,c)$ and contains
a single word $b$ which represents the conjugacy class $c_b$.
\item
Rows for $P_c$ and $P_d$ are defined in a similar fashion.
\end{itemize}
\begin{table}[h]
\centering
\begin{tabular}{c|c|c}
& label & representatives \\
\hline
$P_\varepsilon$ & $(\varepsilon,\varepsilon)$ & $\varepsilon$ \\
\hline
$P_a$ & $\varepsilon$ & $a$ \\
\hline
$P_b$ & $(a,c)$ & $b$ \\
\hline
$P_c$ & $(a,d)$ & $c$ \\
\hline
$P_d$ & $(\varepsilon,b)$ & $d$ \\
\hline
\end{tabular}
\caption{The initial table.}
\label{tab:initial-table}
\end{table}
In Sections \ref{se:processing-even-w} and
\ref{se:processing-odd-w}
we describe how to extend the table.

\subsubsection{Case-I: $w\in\CR^e$}
\label{se:processing-even-w}

Suppose that $w = (w_0,w_1)\in W\cap \CR^e$ and
$w_0,w_1$ are processed, i.e., their representatives
$w_0^\ast,w_1^\ast$ and the sets
$Q(w_0,w_0^\ast),Q(w_1,w_1^\ast)$ are computed.

If the table does not contain a row labelled with
$(w_0^\ast,w_1^\ast)$ and does not contain a row labelled with
$(w_1^\ast,w_0^\ast)$, then $w$ is not conjugate to any word in the table.
In that case we
add a new row labelled with $(w_0^\ast,w_1^\ast)$
with a single representative $w$;
set $w^\ast=w$; and compute $Q(w,w)$ as follows:
\begin{align*}
Q(w_0,w_0)&=Q(w_0,w_0^\ast)^{-1}Q(w_0,w_0^\ast),\\
Q(w_1,w_1)&=Q(w_1,w_1^\ast)^{-1}Q(w_1,w_1^\ast),\\
Q(w_0,w_1)&=Q(w_1,w_0^\ast)^{-1}Q(w_0,w_0^\ast),\\
Q(w_1,w_0)&=Q(w_0,w_1^\ast)^{-1}Q(w_1,w_1^\ast),
\end{align*}
\begin{equation}\label{eq:Qww}
Q(w,w)=
\lift[Q(w_0,w_0) \times Q(w_1,w_1)] \cup
\lift[Q(w_1,w_0) \times Q(w_0,w_1)]a.
\end{equation}
\begin{table}[h]
\centering
\begin{tabular}{c|c}
label & representatives \\
\hline
\hline
$(w_0^\ast,w_1^\ast)$ & $w$ \\
\hline
\end{tabular}
\label{tab:even-row}
\end{table}

If such a row exists, then use Proposition \ref{pr:lifting}
to check if one of its entries $w'=(w_0',w_1')$
is conjugate to $w$. Namely, compute the sets
\begin{align*}
Q(w_0,w_0')&=Q(w_0',w_0^\ast)^{-1}Q(w_0,w_0^\ast),\\
Q(w_1,w_1')&=Q(w_1',w_1^\ast)^{-1}Q(w_1,w_1^\ast),\\
Q(w_0,w_1')&=Q(w_1',w_0^\ast)^{-1}Q(w_0,w_0^\ast),\\
Q(w_1,w_0')&=Q(w_0',w_1^\ast)^{-1}Q(w_1,w_1^\ast),
\end{align*}
$$
Q(w,w')=
\lift[Q(w_0,w_0') \times Q(w_1,w_1')] \cup
\lift[Q(w_1,w_0') \times Q(w_0,w_1')]a.
$$
If $w$ is conjugate to some $w'$ (i.e., if $Q(w,w')\ne\emptyset$),
then set $w^\ast=w'$.
Otherwise, set $w^\ast=w$, add
$w$ to the row, and compute $Q(w,w)$ using \eqref{eq:Qww}.

\subsubsection{Case-II: $w\notin\CR^e$}
\label{se:processing-odd-w}

Assume that $w = (w_0,w_1)a\notin \St(1)$ and
$w_0w_1$ is processed, i.e., its representative
$(w_0w_1)^\ast$ and the set
$Q(w_0w_1,(w_0w_1)^\ast)$ are computed.
If the table does not contain a row labelled with $(w_0w_1)^\ast$,
then add a new row labelled with $(w_0w_1)^\ast$
with a single representative $w$, set $w^\ast=w$, compute the set
$$
Q(w_0w_1, w_0w_1)=Q(w_0w_1,(w_0w_1)^\ast)^{-1}Q(w_0w_1,(w_0w_1)^\ast)
$$
and compute $Q(w,w)$ using \eqref{eq:odd-Q-formula}:
\begin{align}
\label{eq:Qww2}
Q(w,w)=&\lift\Set{(g, w_1 g w^{-1}_1)}{g \in Q(w_0w_1, w_0w_1)} \\
\cup & \lift\Set{(gw_1^{-1}, w_0^{-1}g)}{g \in Q(w_0w_1, w_0w_1)}a.\nonumber
\end{align}

\begin{table}[h]
\centering
\begin{tabular}{c|c}
label & representatives \\
\hline
\hline
$(w_0 w_1)^\ast$ & $w$ \\
\hline
\end{tabular}
\label{tab:odd-row}
\end{table}

If such a row exists, then use Proposition \ref{pr:lifting}
to check if one of its entries $w'=(w_0',w_1')a$
is conjugate to $w$. Namely, compute the sets
\begin{align*}
Q(w_0w_1, w_0'w_1')=&Q(w_0'w_1',(w_0w_1)^\ast)^{-1}Q(w_0w_1,(w_0w_1)^\ast),\\
Q(w,w')=&\lift\Set{(g, w_1' g w^{-1}_1)}{g \in Q(w_0w_1, w_0'w_1')} \\
\cup & \lift\Set{(gw_1^{-1}, (w_0')^{-1}g)}{g \in Q(w_0w_1, w_0'w_1')}a.
\end{align*}
If $w$ is conjugate to $w'$ (i.e., if $Q(w,w')\ne\emptyset$),
then set $w^\ast=w'$.
Otherwise, set $w^\ast=w$, add a new entry $w$ to the row, and compute $Q(w,w)$
using \eqref{eq:Qww2}.

\subsection{The algorithm}

Now we can describe our algorithm to solve the conjugacy problem
in $\Gamma$. We start out with a short sketch of the algorithm.
To determine if $u\sim_\Gamma v$ we
construct the trees $T(u)$ and $T(v)$
and form the set of words
$$
W = \{\mu(s) \mid s\in V(T(u))\}
\cup \{\mu(s) \mid s\in V(T(v))\}
\cup \{\varepsilon,a,b,c,d\}.
$$
Sort $W$ in shortlex ascending order and get a list
$$
\{\varepsilon,a,b,c,d,w_1,w_2,\ldots\}.
$$
Precompute the sets $Q(\varepsilon,\varepsilon),Q(a,a),Q(b,b),Q(c,c),Q(d,d)$
(e.g., as in \cite[Section 4]{Lysenok_Miasnikov_Ushakov:2010})
and form the initial conjugacy table for $\{\varepsilon,a,b,c,d\}$.
Then process words $w_1,w_2,\ldots$,
going from the left to the right and
extend the table as described in Section \ref{se:representatives}.
Finally, output YES if $u$ and $v$ have the same
representatives, i.e.,
$u^\ast=v^\ast$ and NO otherwise.
Below we discuss some important details.

\begin{remark}
Suppose that $w_1,\ldots,w_{i-1}$ are processed and consider $w_i$.
There are two cases. If $w_i\in\St(1)$,
then, as in the proof of Proposition \ref{pr:descendant_length},
its children are strictly shorter
and, hence, have been processed on previous iterations. Such $w_i$
can be processed immediately.

If $w_i\notin\St(1)$, then its child $w_i'$
might have the same length and might be unprocessed.
In that event, we attempt to process $w_i'$.
If it happens so that $w_i'\notin\St(1)$ and its child
$w_i''$ has not been processed, then we
attempt to process $w_i''$.
By Proposition \ref{pr:descendant_length} the children of
$w_i''$ must be shorter than $w_i''$, and, hence, $w_i''$
can be processed. We process $w_i''$, then $w_i'$, then $w_i$.
\end{remark}

\begin{wrapfigure}{r}{0.35\textwidth}
\centering
\includegraphics[scale=0.65]{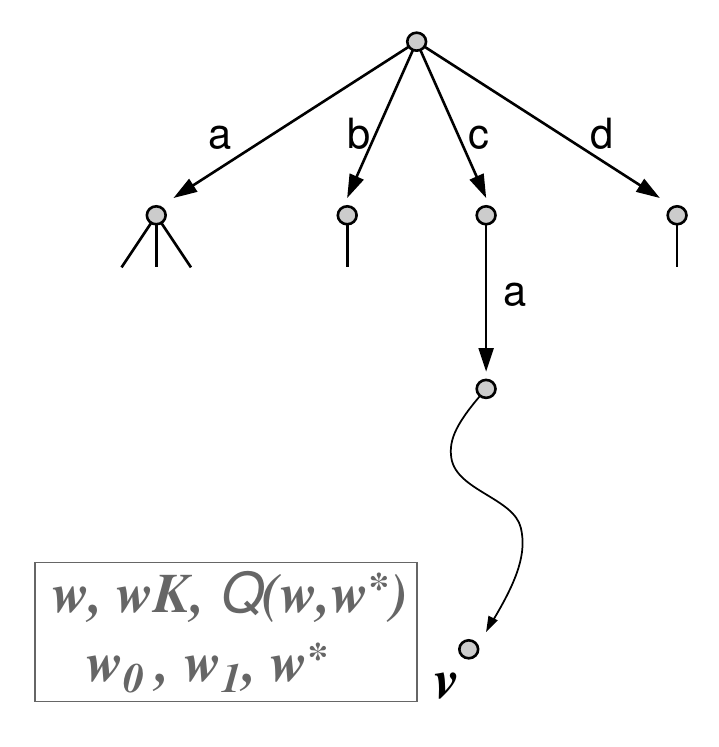}
\label{fig:trie1}
\end{wrapfigure}
We use a trie $\Lambda_1$ to store the set of words $W$ and
another trie $\Lambda_2$ to organize the conjugacy table.
In the end, after all words are processed,
each vertex $v$ in $\Lambda_1$ that represents $w=\mu(v)\in W$
gets the following computed:
\begin{itemize}
\item
The word $w$ and the coset $wK$.
\item
The representative $w^\ast$ of the conjugacy class of $w$
and the set $Q(w,w^\ast)$.
\item
The words $w_0,w_1$ for $w\in\CR^e$, where $\psi(w)=(w_0,w_1)$.
The word $w_0w_1$ for $w\notin\CR^e$, where $\psi(w)=(w_0,w_1)a$,
\end{itemize}
It follows from item (e) of Lemma \ref{le:total-norm_conj} that
the data
\begin{align*}
&\Set{(w,wK,w_0,w_1,w^\ast,Q(w,w^\ast))}{w\in W\cap \CR^e}\\
\cup&
\Set{(w,wK,w_0w_1,w^\ast,Q(w,w^\ast))}{w\in W\setminus \CR^e}.
\end{align*}
requires $O(|u|+|v|)$ space.

The labels $(w_0^\ast,w_1^\ast)$ or $(w_0w_1)^\ast$ in
the conjugacy table are words  over the alphabet $\{a,b,c,d\}$
with comma.
The aforementioned trie $\Lambda_2$ is used for efficient
manipulations with the conjugacy table; it defines the set of all labels
in the table.
Each marked vertex $v$ in $\Lambda_2$ defines the label of a row
and has a comma separated list of at most $256$ representatives.

\begin{proposition}
The described algorithm works in $O(|u|+|v|)$ time.
\end{proposition}

\begin{proof}
Splitting is a straightforward linear time procedure.
Hence, by Lemma \ref{le:total-norm_conj},
it requires $O(|u|+|v|)$ time to enumerate all descendants
of $u$ in $T(u)$ and $v$ in $T(v)$, and form the set $W$.
Using random access machine and tries
we sort words in $W$ in shortlex order
$\varepsilon,a,b,c,d,w_1,w_2,\ldots$ in linear time.
Processing a word $w_i$ requires the following manipulations.
\begin{itemize}
\item
Compute the label $(w_{i0}^\ast,w_{i1}^\ast)$ if $w_i\in\CR^e$
or $(w_{i0} w_{i1})^\ast$ if $w_i\notin\CR^e$.
\item
Find the vertex in $\Lambda_2$
corresponding to $(w_{i0}^\ast,w_{i1}^\ast)$
or $w_{i0}^\ast w_{i1}^\ast$ and access the list of at most
$256$ representatives;
or determine that such a row does not exist
and add a new row.
\item
For a representative $w'$ compute $Q(w,w')$
using formulae from Sections
\ref{se:processing-even-w} and \ref{se:processing-odd-w}.
\end{itemize}
Each of these steps can be done in $O(|w_i|)$ steps.
It follows from Lemma \ref{le:total-norm_conj} that
it requires $O(|u|+|v|)$ time to process all words in $W$.
\end{proof}

\begin{theorem}
Given $w_1,\ldots,w_k\in\Gamma$, we can decide if
$w_i\sim_\Gamma w_j$ for some $i\ne j$ in $O(|w_1|+\ldots+|w_k|)$ time.
\end{theorem}

\begin{proof}
Compute the set $W=\bigcup_i \{\mu(s) \mid s\in V(T(w_i))\}$,
process $W$ as above, and decide if the list
$w_1^\ast,\ldots,w_k^\ast$ has two equal words.
\end{proof}

\section{Conjugacy search problem}
\label{se:conjugator}

In this section we shortly discuss
the problem of computing a conjugator for
a pair of conjugate elements in $\Gamma$.

\begin{proposition}\label{pr:lift1}
Let $x_0,x_1\in\CR$ be such that $(x_0,x_1)\in\psi(\St(1))$.
Then there exists $x\in\CR^e$ satisfying the following conditions:
$$
\psi(x)=(x_0,x_1)
\mbox{ in }\Gamma\times\Gamma
\ \ \mbox{ and }\ \
|x|\le 2(|x_0|+|x_1|)+10.
$$
Furthermore, such $x$ can be computed in $O(|x_0|+|x_1|)$ time.
\end{proposition}

\begin{proof}
Consider maps $\tau_0,\tau_1:\{a,b,c,d\}^\ast\to\{a,b,c,d\}^\ast$
defined as follows:
\begin{align*}
\tau_0(a)&= c  &\tau_1(a)&=aca\\
\tau_0(b)&=ada &\tau_1(b)&=d\\
\tau_0(c)&=aba &\tau_1(c)&=b\\
\tau_0(d)&=aca &\tau_1(d)&=c.
\end{align*}
It is easy to check that $\tau_0,\tau_1$ have the following properties:
\begin{itemize}
\item
$|\tau_i(w)|\le 2|w|+1$ for any $w\in\CR$.
\item
For any $x_0\in\CR$ there exists $\delta_0\in\{a,d\}^\ast$ such that
$\psi(\tau_0(x_0))=(x_0,\delta_0)$ in $\Gamma\times\Gamma$.
Since $(ad)^4=1$ we may assume that $|\delta_0|\le 4$.
\item
For any $x_1\in\CR$ there exists $\delta_1\in\{a,d\}^\ast$
of length at most $4$ such that
$\psi(\tau_1(x_1))=(\delta_1,x_1)$ in $\Gamma\times\Gamma$.
\end{itemize}
Let $z_0=\tau_0(x_0)$ and $\delta_0\in\{a,d\}^\ast$ a word of length
at most $4$ satisfying $\psi(z_0)=(x_0,\delta_0)$.
Let $z_1=\tau_1(\delta_0^{-1}x_1)$ and $\delta_1\in\{a,d\}^\ast$
a word of length  at most $4$ satisfying
$\psi(z_1)=(\delta_1,\delta_0^{-1}x_1)$.
By construction, the word $x=z_0z_1$ satisfies the following:
$$
\psi(x)=(x_0\delta_1,x_1)\ \ \mbox{and}\ \ |x|\le 2(|x_0|+|x_1|)+10.
$$
Finally, notice that $(\delta_1,\varepsilon)\in\psi(\St(1))$
since $(x_0,x_1),(x_0\delta_1,x_1)\in\psi(\St(1))$.
The subgroup $\psi(\St(1))$ has index $8$ in $\Gamma\times\Gamma$
and its Schreier graph relative to the generating set
$(1,a),(a,1),(1,b),(b,1),(1,d),(d,1)$
is shown in Figure \ref{fig:schreier-lift}.
\begin{figure}[h]
\includegraphics[scale=0.5]{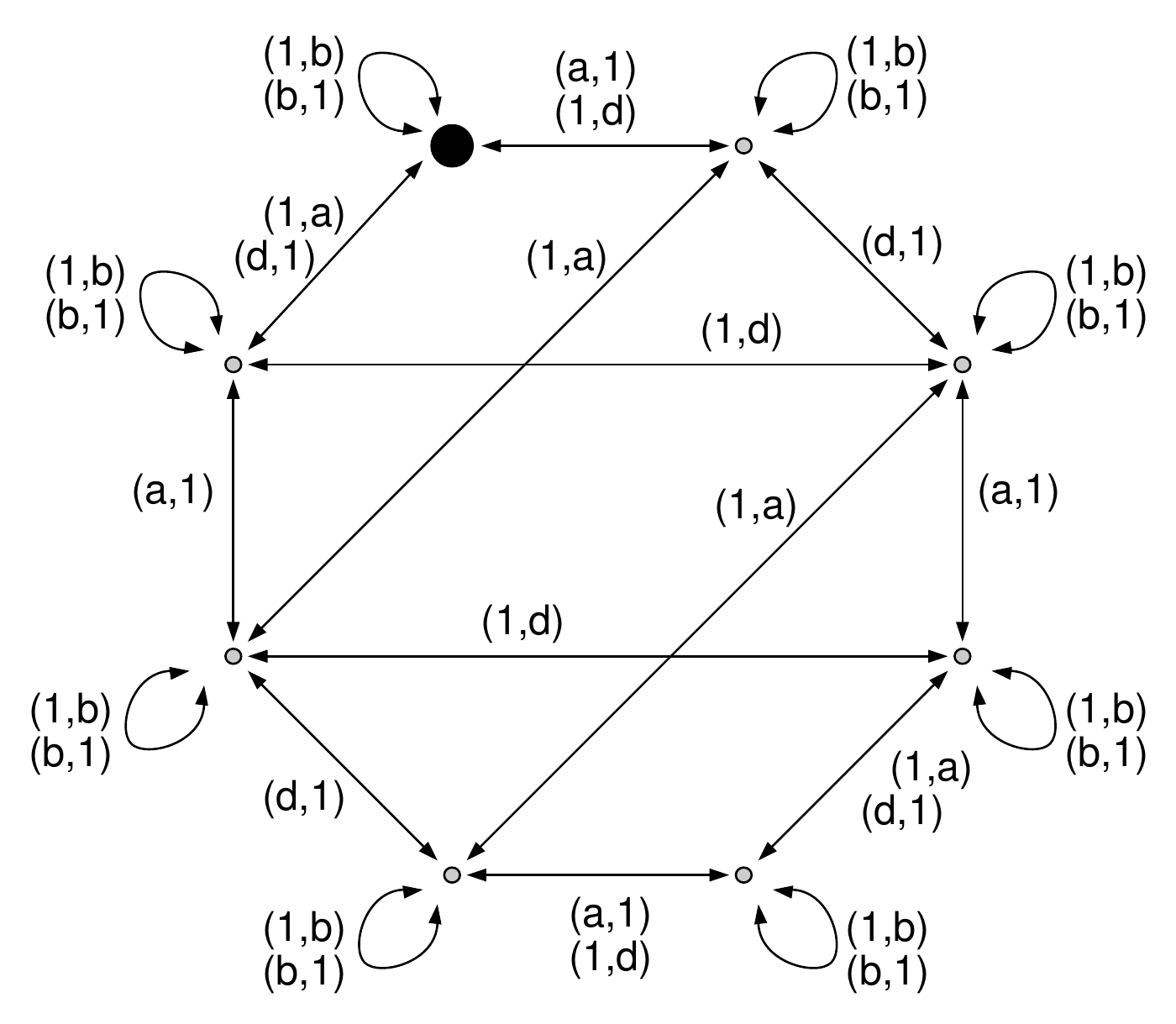}
\caption{\label{fig:schreier-lift}
The
Schreier graph of $\psi(\St(1))$ in $\Gamma\times\Gamma$.}
\end{figure}
Using the Schreier graph of $\psi(\St(1))$ we deduce that
$\delta_1=\varepsilon$. Thus, $x$ is the required word.
\end{proof}

%

Next, we prove that the length of a conjugator
for conjugate elements $u,v\in\CR$ can be bounded
by a polynomial of $|u|+|v|$.
For $u\in\CR$ define $h(u)$ to be the \emph{height} of
the tree $T(u)$, i.e., the longest distance from the
root to a leaf. For a pair $u,v\in\CR$ define
$$
h(u,v) = \max(h(u),h(v)).
$$
By Proposition \ref{pr:norm-bound},
$$
h(u,v) = \log_{1.22}(|u|+|v|) + O(1).
$$
For $u,v\in\CR$ and $g\in Q(u,v)$ define
$$
L(u,v,g) = \min \{|x| \mid x^{-1}ux=v \mbox{ and }xK = g\}
$$
and for $n\in\MN$ define
$$
L(n) = \max_{\substack{g\in Q(u,v)\\h(u,v)\le n\\}} L(u,v,g),
$$
that can be called the \emph{conjugacy-growth function} relative
to the splitting tree height.

\begin{proposition}\label{pr:conjugator_length}
$L(h(u,v))$ is $O((|u|+|v|)^{8})$.
\end{proposition}

\begin{proof}
Fix a pair of conjugate elements $u,v\in\CR$.
Consider the four cases of Proposition \ref{pr:conj-systems}.
\begin{itemize}

\item[(1)]
If $u=(u_0,u_1),v=(v_0,v_1)\in\CR^e$
is a case of item (1) of Proposition
\ref{pr:conj-systems},
then $x_0^{-1} u_0x_0=v_0$
and $x_1^{-1} u_1x_1=v_1$
for some $x_0,x_1\in\CR$ satisfying
$
|x_i|\le L(h(u_i,v_i)) \le L(h(u,v)-1)
$
for each $i=0,1$.
By Proposition \ref{pr:lift1} we have
$$
|x|\le 2(|x_0|+|x_1|)+10.
$$

\item[(2)]
If $u=(u_0,u_1),v=(v_0,v_1)\in\CR^e$
is a case of item (2) of Proposition
\ref{pr:conj-systems}, then the same argument
shows that a conjugator $xa$ satisfies
$$
|xa|\le 2(|x_0|+|x_1|)+11
$$
for some $x_0,x_1\in\CR$ of length at most $L(h(u,v)-1)$.
\item[(3)]
If $u=(u_0,u_1)a,v=(v_0,v_1)a\in\CR\setminus\CR^e$
is a case of item (3) of Proposition \ref{pr:conj-systems},
then
$u_0u_1 = x_0^{-1} v_0v_1 x_0$ for some $x_0\in\CR$
satisfying
$$
|x_0|\le L(h(u_0u_1,v_0v_1)) = L(h(u,v)-1).
$$
That defines the element $x_1=v_1x_0u_1^{-1}$ and,
by Proposition \ref{pr:lift1},
$$
|x|\le 2(|x_0|+|x_0|+|u_1|+|v_1|)+10.
$$

\item[(4)]
If $u=(u_0,u_1)a,v=(v_0,v_1)a\in\CR\setminus\CR^e$
is a case of item (4) of Proposition \ref{pr:conj-systems},
then
$u_1u_0 = x_0^{-1} v_0v_1 x_0$ for some $x_0\in\CR$
satisfying
$$
|x_0|\le L(h(u_1u_0,v_0v_1)) = L(h(u,v)-1).
$$
That defines the element $x_1=v_1x_0u_0^{-1}$ and,
by Proposition \ref{pr:lift1},
$$
|xa|\le 2(|x_0|+|x_0|+|u_0|+|v_1|)+11.
$$
\end{itemize}
In each case the following upper bound holds
\begin{align*}
L(h(u,v))
&\le 4 L(h(u,v)-1))+4(|u|+|v|)+11 \\
&\le
4^{h(u,v)} L(1) + 4^{h(u,v)}\sum_{s\in V(T(u))\cup V(T(v))} |\mu(s)|+11\sum_{i=0}^{h(u,v)-2}4^i\\
&=
4^{\log_{1.22}(|u|+|v|)+O(1)}(O(1)+800(|u|+|v|))\\
&=O((|u|+|v|)^{8}).
\qedhere
\end{align*}
\end{proof}

Notice that
the proof of Proposition \ref{pr:conjugator_length}
does not simply show the existence of a short (polynomial size)
conjugator for a pair of conjugate elements $u,v \in\Gamma$,
it also describes a procedure to compute a short conjugator.
That implies the following theorem.

\begin{theorem}
There exists a polynomial time algorithm
to compute a conjugator $x\in\CR$ for a
pair of conjugate elements $u,v\in\CR$.
\end{theorem}


\begin{thebibliography}{10}

\bibitem{Bartholdi:1998}
L.~{Bartholdi}.
\newblock {The growth of Grigorchuk's torsion group}.
\newblock {\em Internat. Math. Res. Notices}, 20:1049--1054, 1998.

\bibitem{Bartholdi-Groth-Lysenok:2019}
L.~{Bartholdi}, T.~{Groth}, and I.~{Lysenok}.
\newblock {Commutator width in the first Grigorchuk group}.
\newblock preprint. Available at \url{https://arxiv.org/abs/1710.05706}, 2019.

\bibitem{Booth:1980}
K.~S. Booth.
\newblock Lexicographically least circular substrings.
\newblock {\em Information Processing Letters}, 10(4):240--242, 1980.

\bibitem{Epstein-Holt:2006}
D.~{Epstein} and D.~{Holt}.
\newblock The linearity of the conjugacy problem in word-hyperbolic groups.
\newblock {\em Int. J. Algebra Comput.}, 16:287--305, 2006.

\bibitem{GZ}
M.~{Garzon} and Y.~{Zalcstein}.
\newblock {The complexity of Grigorchuk groups with application to
  cryptography}.
\newblock {\em Theoret. Comput. Sci.}, 88:83--98, 1991.

\bibitem{Grigorchuk:1980}
R.~I. {Grigorchuk}.
\newblock {Burnside's problem on periodic groups}.
\newblock {\em Funct. Anal. Appl.}, 14:41--43, 1980.

\bibitem{Grigorchuk:1984}
R.~I. {Grigorchuk}.
\newblock {Degrees of growth of finitely generated groups and the theory of
  invariant means}.
\newblock {\em (Russian) Izv. Akad. Nauk SSSR Ser. Mat.}, 48(5):939--985, 1984.

\bibitem{Grigorchuk:2006}
R.~I. {Grigorchuk}.
\newblock {Solved and unsolved problems around one group}.
\newblock In {\em Infinite Groups: Geometric, Combinatorial and Dynamical
  Aspects}, volume 248 of {\em Progress in Mathematics}, pages 117--218.
  Birkhauser Basel, 2006.

\bibitem{Grigorchuk-Wilson:2003}
R.~I. {Grigorchuk} and J.~{Wilson}.
\newblock A structural property concerning abstract commensurability of
  subgroups.
\newblock {\em Journal of the London Mathematical Society}, 68(3):671--682,
  2003.

\bibitem{Harpe}
Pierre de~la {Harpe}.
\newblock {\em Topics in geometric group theory}.
\newblock The University of Chicago Press, 2000.

\bibitem{KMP}
D.~{Knuth}, J.~H. {Morris}, and V.~{Pratt}.
\newblock Fast pattern matching in strings.
\newblock {\em SIAM J. Comput.}, 6:323--350, 1977.

\bibitem{Leonov:1998}
Yu.~G. {Leonov}.
\newblock {The conjugacy problem in a class of 2-groups}.
\newblock {\em Mat. Zametki}, 64:573--583, 1998.

\bibitem{Lysenok:1985}
I.~{Lysenok}.
\newblock A system of defining relations for a grigorchuk group.
\newblock {\em Math. Notes}, 38(4):503--516, 1985.

\bibitem{Lysenok_Miasnikov_Ushakov:2010}
I.~{Lysenok}, A.~G. {Miasnikov}, and A.~{Ushakov}.
\newblock {The conjugacy problem in the Grigorchuk group is polynomial time
  decidable}.
\newblock {\em Groups Geom. Dyn.}, 4(4):813--833, 2010.

\bibitem{Lysenok_Miasnikov_Ushakov:2016}
I.~{Lysenok}, A.~G. {Miasnikov}, and A.~{Ushakov}.
\newblock {Quadratic equations in the Grigorchuk group}.
\newblock {\em Groups Geom. Dyn.}, 10(1):201--239, 2016.

\bibitem{Macdonald-Myasnikov-Nikolaev-Vassileva:2015}
J.~{Macdonald}, A.~{Myasnikov}, A.~{Nikolaev}, and S.~{Vassileva}.
\newblock Logspace and compressed-word computations in nilpotent groups.
\newblock arXiv:1503.03888 [math.GR], 2015.

\bibitem{MU5}
A.~G. {Miasnikov} and A.~{Ushakov}.
\newblock {Random subgroups and analysis of the length-based and quotient
  attacks}.
\newblock {\em J. Math. Crypt.}, 2:29--61, 2008.

\bibitem{P}
G.~{Petrides}.
\newblock Cryptanalysis of the public key cryptosystem based on the word
  problem on the grigorchuk groups.
\newblock In {\em 9th IMA International Conference on Cryptography and Coding},
  volume 2898 of {\em Lecture Notes Comp. Sc.}, pages 234--244. Springer, 2003.

\bibitem{Rozhkov:1998}
A.~V. {Rozhkov}.
\newblock {The conjugacy problem in an automorphism group of an infinite tree}.
\newblock {\em Mat. Zametki}, 64:592--597, 1998.

\end{thebibliography}
 \end{document}